\documentclass[reqno,draft]{amsart}
\usepackage{amsfonts}%
\usepackage{amssymb}
\usepackage{mathrsfs}
\usepackage{cite}
\usepackage{color}
\allowdisplaybreaks
\date{\today}

 \numberwithin{equation}{section}

\newcommand{\dv}{\mathrm{div}\,}
\newcommand{\cl}{\mathrm{curl}\,}
\newtheorem{Theorem}{Theorem}[section]
\newtheorem{Lemma}{Lemma}[section]
\newtheorem{Proposition}{Proposition}[section]

\newtheorem{Remark}{Remark}[section]

\begin{document}

\title[Regularity criteria   for the  Boussinesq system]
 {Regularity criteria and uniform   estimates for the Boussinesq system
  with  temperature- dependent viscosity and  thermal diffusivity}

\author{Jishan Fan}
\address{ Department of Applied Mathematics,
 Nanjing Forestry University, Nanjing 210037, P.R.China}
\email{fanjishan@njfu.edu.cn}

 \author[Fucai Li]{Fucai Li$^*$}
\address{Department of Mathematics, Nanjing University, Nanjing
 210093, P.R. China}
 \email{fli@nju.edu.cn}
 \thanks{$^*$Corresponding author}

 \author[G. Nakamura]{Gen Nakamura}
\address{Department of Mathematics, Inha University, Incheon 402-751, Korea}
 \email{nakamuragenn@gmail.com}

\begin{abstract}
In this paper we establish some regularity criteria for the 3D Boussinesq system with
the temperature-dependent viscosity and thermal diffusivity. We also obtain some uniform  estimates for the
corresponding 2D case when the fluid viscosity coefficient is a positive constant.
\end{abstract}

\keywords{Boussinesq system, temperature-dependent viscosity, thermal diffusivity, regularity criterion, uniform regularity estimates}
\subjclass[2010]{35Q30, 76D03, 76D09}

\maketitle

\section{Introduction}
In this paper we establish some regularity criteria for the 3D Boussinesq system with
the temperature-dependent viscosity and thermal diffusivity (\!\cite{DR,Mi}):
\begin{align}
&\partial_tu+u\cdot\nabla u-\dv(\mu(\theta)(\nabla
u+\nabla u^\mathrm{T}))+\nabla\pi=\theta e_3,\label{baa}\\
&\partial_t\theta+u\cdot\nabla\theta-\dv(\kappa(\theta)\nabla\theta)=0,\label{bac}\\
&\dv u=0,\label{bab}\\
&(u,\theta)(x,0)=(u_0,\theta_0)(x),\ \ \ x\in\mathbb{R}^3.\label{bad}
\end{align}
Here the unknowns $u,\pi$, and $\theta$ denote the velocity, pressure and the
temperature of the fluid, respectively. $e_3: =(0,0,1)^\mathrm{T}$   is the unit vector in the vertical
$x_3$-direction.  The kinematic viscosity $\mu(\theta)$ and the  thermal conductivity $\kappa(\theta)$ are generally depending on
the temperature $\theta$.  Throughout this paper, we assume that $\mu(\theta)$ and $\kappa(\theta)$ are
smooth functions and satisfy
\begin{equation}
0<  {C}^{-1}_1\leq\mu(\theta), \kappa(\theta)\leq C_1<\infty\ \ \text{when} \ \
|\theta|\leq C_2 \label{bae}
\end{equation}
for some positive constant $C_1>1$ and $C_2>0$.  Boussinseq type system \eqref{baa}-\eqref{bac} are
 used to  model geophysical flows such as atmospheric fronts and
ocean circulations (see \cite{Ma,Pe}).

Due to its physical importance and mathematical complexity,
there are a lot of research papers concerned on   the Boussinseq system when both $\kappa(\theta)$ and $\mu(\theta)$ are independent of $\theta$, see
 \cite{AH,ACWa,ACWb,BS,Ch,CN,CW,DPa,DPb,FNW,FZ,HKRa,HKRb,HL,HK,HR,LLT,LPZ,MX} and the references cited therein.
When  $\kappa(\theta)$ or $\mu(\theta)$  depends on $\theta$, only a few results are aviable, see \cite{A,DGa,DGb,LB,LBb,WZ}.
In \cite{LBb,LB},  Lorca and Boldrini  obtained the global existence of weak
solution with small initial data and the local existence of strong solution for general
data to the system \eqref{baa}-\eqref{bab}, see also \cite{DGa,DGb}. In \cite{A}, Abidi obtained the
global existence and uniqueness result to the system \eqref{baa}-\eqref{bab} with $k=0$ in some critical spaces  when
the initial temperature is small.
  Recently, Wang and Zhang \cite{WZ} obtain the global existence of smooth
solutions to the problem \eqref{baa}-\eqref{bad} with general initial data in $\mathbb{R}^2$. However, whether or not
the smooth solution to the 3D boussinesq system \eqref{baa}-\eqref{bab} with general initial data blows up in
finite time is still a big open problem. In this paper we provide some regularity criteria related to this topic.
 We will prove

\begin{Theorem}\label{Ta}
Let $\kappa(\theta)\equiv1$  and $\mu(\theta)$ satisfy \eqref{bae} and $u_0,\theta_0\in
H^2(\mathbb{R}^3)$ with $\dv u_0=0$ in $\mathbb{R}^3$. Suppose that there exists a $T>0$ such that
one of the following two conditions holds:
\begin{align}
&u\in L^p(0,T;L^q(\mathbb{R}^3))\ \ with\ \
\frac{2}{p}+\frac{3}{q}=1,\ 3<q\leq\infty,\label{baf}\\
&\nabla u\in L^p(0,T;L^q(\mathbb{R}^3))\ \ with\ \
\frac{2}{p}+\frac{3}{q}=2,\ 3/2<q\leq\infty.\label{bag}
\end{align}
Then the solution $(u,\theta)$ of the problem \eqref{baa}-\eqref{bad} satisfies
\begin{equation}
u,\theta\in L^\infty(0,T;H^2(\mathbb{R}^3))\cap L^2(0,T;H^3(\mathbb{R}^3)).\label{bah}
\end{equation}
\end{Theorem}

\begin{Theorem}\label{Tb}
Let $\mu(\theta)$ and $\kappa(\theta)$ satisfy (\ref{bae}). Let
$u_0,\theta_0\in H^2(\mathbb{R}^3)$ with $\dv u_0=0$ in $\mathbb{R}^3$. Assume
that (\ref{baf}) holds   with $3<q<\infty$, then the solution $(u,\theta)$ of the problem \eqref{baa}-\eqref{bad}   satisfies  (\ref{bah}).

\end{Theorem}

%

\begin{Remark}
We can also obtain the   regularity criteria \eqref{baf} or  \eqref{bag}
 for the Cahn- \linebreak Hilliard-Navier-Stokes system \cite{HH,GP}:
\begin{align*}
&\partial_tu+u\cdot\nabla u-\nabla\cdot\mu(\phi)(\nabla
u+\nabla u^\mathrm{T})+\nabla\pi=\eta\nabla\phi,\ \
 \dv u=0,\\
&\partial_t\phi+u\cdot\nabla\phi-\Delta\eta=0,\\
&\eta=-\Delta\phi+f'(\phi),\ f(\phi)=\frac{1}{4}(\phi^2-1)^2,\\
&(u,\phi)(x,0)=(u_0,\phi_0)(x),\ \ \ x\in\mathbb{R}^3.
\end{align*}
Since the proof is similar to that of  Theorem \ref{Ta}, we omit the
details here.
\end{Remark}

Because the local-in-time well-posedness
can be obtained in a standard way (see Propositions \ref{le2.5} and \ref{le2.6} below), to prove Theorems \ref{Ta} and \ref{Tb},
 it suffices to obtain a priori estimates under the conditions \eqref{baf} or \eqref{bag}.
Due to  the equations \eqref{baa} and \eqref{bac} are strongly coupled, it is difficult to
obtain the  $L^p$ estimates of $\nabla \theta$ and the H\"{o}lder continuity of $\theta$, we shall
use some results obtained in \cite{Am,WZ}, Sobolev imbedding, and the ingenious
applications of H\"{o}lder and  Gagliardo-Nirenberg inequalities  to obtain the desired estimates.

\medskip
Next, we consider the Boussinesq system \eqref{baa}-\eqref{bac} in two-dimensional spatial space $\mathbb{R}^2$ and
 let  the kinematic viscosity $\mu(\theta)$ be  a  positive constant, i.e,
$\mu(\theta)\equiv \epsilon>0$. In this case the system reads
\begin{align}
&\partial_tu+u\cdot\nabla u-\epsilon\Delta u+\nabla\pi=\theta
e_2,\label{bai}\\
&\partial_t\theta+u\cdot\nabla\theta-\dv(\kappa(\theta)\nabla\theta)=0,\label{bak}\\
&\dv u=0,\label{baj}\\
&(u,\theta)(x,0)=(u_0,\theta_0)(x),\ \ \ x\in\mathbb{R}^2.\label{bal}
\end{align}
Here $e_2:=(0,1)^\mathrm{T}$. As before, the    thermal conductivity $\kappa(\theta)$ is assumed  to be
  smooth and satisfy
\begin{equation}0<  {C}^{-1}_1\leq \kappa(\theta)\leq C_1<\infty\ \ \text{when} \ \
|\theta|\leq C_2 \label{bam}
\end{equation}
for some positive constant $C_1>1$ and $C_2>0$.

Notice that in \cite{WZ}  Wang and Zhang established   the global existence of classical
solutions for the problem (\ref{bai})-(\ref{bal}) with general initial data and fixed $\epsilon>0$.
In  this paper we give some  uniform-in-$\epsilon$ estimates for smooth solution to the
problem (\ref{bai})-(\ref{bal}). Our result reads

\begin{Theorem}\label{Tc}
Let $\kappa(\theta)$ satisfy (\ref{bam}) and $\epsilon\in(0,1)$. Let $\theta_0\in H^2(\mathbb{R}^2)$ and
$u_0\in H^3(\mathbb{R}^2)$ with $\dv u_0=0$ in $\mathbb{R}^2$. Then, for any given $T>0$, the solution
$(u,\theta)$ to the problem \eqref{bai}-\eqref{bal} satisfies
\begin{equation}
\|u\|_{L^\infty(0,T;H^3(\mathbb{R}^2))}+\|\theta\|_{L^\infty(0,T;H^2(\mathbb{R}^2))}
+\|\theta\|_{L^2(0,T;H^3(\mathbb{R}^2))}\leq C,\label{ban}
\end{equation}
where  $C$ is  a positive constant independent of
$\epsilon$.
  \end{Theorem}

Finally,  considering the following 2D Boussinesq system
\begin{align}
&\partial_tu+u\cdot\nabla u-\Delta u+\nabla\pi=\theta
e_2,\label{bao}\\
&\partial_t\theta+u\cdot\nabla\theta-\epsilon\dv(\kappa(\theta)\nabla\theta)=0,\label{baq}\\
&\dv u=0,\label{bap}\\
&(u,\theta)(x,0)=(u_0,\theta_0)(x),\ \  \ x\in
\mathbb{R}^2,\label{bar}
\end{align}
we can obtain a similar result on   uniform-in-$\epsilon$ estimates for smooth solution to the
problem (\ref{bao})-(\ref{bar}).
\begin{Theorem}\label{Td}
Let $\kappa(\theta)$ satisfy (\ref{bam}) and $\epsilon\in(0,1)$. Let
$u_0,\theta_0\in H^2(\mathbb{R}^2)$ and $\dv u_0=0$ in $\mathbb{R}^2$. Then, for any given $T>0$, the solution
$(u,\theta)$ to the problem \eqref{bao}-\eqref{bar} satisfies
\begin{equation}
\|u\|_{L^\infty(0,T;H^2(\mathbb{R}^2))}+\|u\|_{L^2(0,T;H^3(\mathbb{R}^2))}
+\|\theta\|_{L^\infty(0,T;H^2(\mathbb{R}^2))}\leq C,\label{bas}
\end{equation}
where  $C$ is  a positive constant independent of $\epsilon$.
  \end{Theorem}

\begin{Remark}\label{Rc}
 In   \cite{Ch}, Chae has proved  similar results to Theorems \ref{Tc} and \ref{Td} when $\kappa(\theta) \equiv 1$, it seems impossible to
 apply   his method   directly to our case. In fact,   if we follow his arguments we shall
encounter some unpleasant terms involved $\kappa(\theta)$ which are out of control.
\end{Remark}

We give some comments on the proofs of Theorems \ref{Tc} and \ref{Td}.   Because the global-in-time well-posedness results
 (see Propositions \ref{le2.7} and \ref{le2.8} below), to complete our proofs, we only
need to prove the \emph{a priori} estimates \eqref{ban} and \eqref{bas}.   We shall employ an elaborate nonlinear
energy method to obtain these desired bounds. More precisely, we first use maximum principle
to obtain  $L^\infty$ estimates of $\theta$. Next, we  derive an
energy estimate based on $L^2$ energy. Finally, we use Amann's $L^p$  estimates on uniform parabolic equations,
 Sobolev imbedding, bilinear commutator estimates, logarithmic Sobolev inequality,  and Gagliardo-Nirenberg inequalities  to obtain
the desired higher order spatial  estimates on  $\theta$ and  $u$.

This paper is organized as follows. In Section 2, we   recall some basic inequalities
and state the local/global existence results on the problems \eqref{baa}-\eqref{bad}, \eqref{bai}-\eqref{bal}, and
 \eqref{bao}-\eqref{bar}. The proofs of Theorems  \ref{Ta}, \ref{Tb}, \ref{Tc}, and    \ref{Td} are presented in the subsequent four sections.

\medskip
\section{Preliminary}

In this section, we  first recall some basic
inequalities  which shall  be used frequently.

\begin{Lemma}[Logarithmic Sobolev inequality \cite{KOT}] \label{Log}  For all $u\in H^{s}(\mathbb{R}^d) $
with $ s>1+d/2$, there exists a constant $C$ such that the following
estimate holds
\begin{equation}
\|\nabla v\|_{L^\infty}\leq
C(1+\|\dv v\|_{L^\infty}+\|\cl v\|_{L^\infty})(1+\log(e+\|v\|_{H^s})).\label{log}
\end{equation}
\end{Lemma}

 \begin{Lemma}[Gagliardo-Nirenberg inequality \cite{Ga,Ni}]\label{GN}
Let $v\in W^{k,r}(\mathbb{R}^d)\cap L^q(\mathbb{R}^d)$, $1\leq q,r
\leq \infty$. Then  the following inequalities hold
\begin{equation}\label{gn}
  \|D^iv\|_{L^p}\leq M_0 \|D^k v\|^\alpha_{L^r}\|v\|^{1-\alpha}_{L^q}, \ \ \ \forall \,\, 0\leq i<k,
\end{equation}
where
\begin{align*}
    \frac{1}{p}=\frac{i}{N}+\alpha\Big(\frac{1}{r}-\frac{k}{d}\Big)+(1-\alpha)\frac{1}{q},
\end{align*}
for all $\alpha$ in the interval
\begin{align*}
   \frac{i}{k}\leq \alpha \leq 1.
\end{align*}
The constant $M_0$ depending only on $d,m,j,q,r$ and $\alpha$, with
the following exceptional case:
\begin{enumerate}
  \item  If\, $i=0$, $rk<d,q=\infty$ then we make the additional assumption that either $v$ tends to zero
  at infinity or $v\in L^{\tilde q}(\mathbb{R}^d)$ for some finite $\tilde q>0$.
  \item \,If $1<r<\infty$, and $k-i-d/r$ is a non negative integer then \eqref{gn} holds  only for $\alpha$ satisfying
  $i/k\leq \alpha <1$.
\end{enumerate}
\end{Lemma}

We define the   operator $\Lambda:=(-\Delta)^{1/2}$   via the
Fourier transform
 $$
\widehat{\Lambda f}(\xi)=|\xi|\hat f (\xi).
$$
Generally, we define $\Lambda^s f$ for $s\in \mathbb{R}$ as
$$
\widehat{\Lambda^s f}(\xi)=|\xi|^{s}\hat f (\xi).
$$
For $s\in \mathbb{R}$, we define
$$
\|f\|_{\dot{H}^s}: =\|\Lambda^s f\|_{L^2}=\Big(\int_{\mathbb{R}^3}
|\xi|^{2s}|\hat f(\xi)|^2 {\,d}\xi\Big)^{1/2}
$$
and the homogeneous Sobolev space $\dot{H}^s(\mathbb{R}^3):=\{f\in
\mathcal{S}'(\mathbb{R}^3): \|f\|_{\dot{H}^s}<\infty\}.$ Similar,
the Sobolev space $H^{s,p}(\mathbb{R}^3)$ is equipped with the norm
$$
\|f\|_{\dot{H}^{s,p}}:=\|\Lambda ^s f\|_{L^p}.
$$

Now we recall the following bilinear commutator and the product
estimates.
\begin{Lemma} [\!\cite{KP}] \label{Bilin}
Let $s>0,1<p<\infty$. Assume that $f,g\in
\dot{H}^{s,p}(\mathbb{R}^d)$, then there exists a constant $C$,
independent of $f$ and $g$, such that,
\begin{align}
&\|\Lambda^s(f g)-f\Lambda^sg\|_{L^p}\leq C(\|\nabla
f\|_{L^{p_1}}\|\Lambda^{s-1}g\|_{L^{q_1}}+\|\Lambda^sf\|_{L^{p_2}}
\|g\|_{L^{q_2}}),\label{kpa}\\
&\|\Lambda^s(f g)\|_{L^p}\leq
C(\|f\|_{L^{p_1}}\|\Lambda^sg\|_{L^{q_1}}+\|\Lambda^sf\|_{L^{p_2}}
\|g\|_{L^{q_2}}),\label{kpb}
\end{align}
where $p_1, p_2\in (1,\infty]$ satisfies  $$\frac{1}{p}=\frac{1}{p_1}+\frac{1}{q_1}
=\frac{1}{p_2}+\frac{1}{q_2}.$$
\end{Lemma}

\begin{Lemma}[\!\cite{Tr}]  Let $s>0, 1<p<\infty$, and $  f\in \dot{H}^{s,p}(\mathbb{R}^d)\cap
L^\infty(\mathbb{R}^d)$. Assume that $F(\cdot)$ is a smooth function
on $\mathbb{R}$ with $F(0)=0$. Then we have
\begin{equation}
\|F(f)\|_{\dot{H}^{s,p}}\leq
C(M)(1+\|f\|_{L^\infty})^{[s]+1}\|f\|_{\dot{H}^{s,p}},\label{biln}
\end{equation}
where  the constant  $C(M)$  depends on
$M:=\sup\limits_{k\leq[s]+2,|t|\leq\|f\|_{L^\infty}}\|F^{(k)}(t)\|_{L^\infty}$.
\end{Lemma}

Finally, we state  the local/global well-posedness  results on the problems \eqref{baa}-\eqref{bad}, \eqref{bai}-\eqref{bal}, and
 \eqref{bao}-\eqref{bar}.

\begin{Proposition}\label{le2.5}
Let $\kappa(\theta)\equiv1$ and $\mu(\theta)$ satisfy \eqref{bae} and $u_0,\theta_0\in H^2(\mathbb{R}^3)$ with $\dv u_0=0$ in $\mathbb{R}^3$. Then the problem \eqref{baa}-\eqref{bad} has a unique local strong solution $(u,\theta)$ satisfying
\begin{equation}
(u,\theta)\in L^\infty(0,T;H^2(\mathbb{R}^3))\cap L^2(0,T;H^3(\mathbb{R}^3))\label{2.6}
\end{equation}
for some positive constant $T>0$.
\end{Proposition}

\begin{proof}
We can prove it by the standard Galerkin method. Since the key step is to give a priori estimates \eqref{2.6},
which are very similar to that of our proofs on the regularity criteria, and thus we omit the details here.
\end{proof}

Similarly, we can obtain the following  proposition.

\begin{Proposition}\label{le2.6}
Let $\mu(\theta)$ and $\kappa(\theta)$ satisfy \eqref{bae}. Let $u_0,\theta_0\in H^2(\mathbb{R}^3)$ with $\dv u_0=0$ in $\mathbb{R}^3$.
 Then  the problem \eqref{baa}-\eqref{bad} has a unique local strong solution $(u,\theta)$ satisfying \eqref{2.6} for some positive constant $T>0$.
\end{Proposition}

For the problem \eqref{bai}-\eqref{bal}, we have
\begin{Proposition}\label{le2.7}
Let $\kappa(\theta)$ satisfy \eqref{bam} and $\epsilon\in(0,1)$. Let $\theta_0\in H^2(\mathbb{R}^2)$ and $u_0\in H^3$ with $\dv u_0=0$ in $\mathbb{R}^2$.
Then, for fixed $\epsilon>0$, the problem \eqref{bai}-\eqref{bal} has a unique global solution $(u,\theta)$ satisfying
\begin{equation}
\|u\|_{L^\infty(0,T;H^3(\mathbb{R}^2))}+\|\theta\|_{L^\infty(0,T;H^2(\mathbb{R}^2))}+\|\theta\|_{L^2(0,T;H^3(\mathbb{R}^2))}\leq C(\epsilon)\label{2.7}
\end{equation}
for some positive constant $C(\epsilon)$.
\end{Proposition}

\begin{proof}
The local existence of solution  can proved by the standard Galerkin method. To obtain the global existence of solution,
it suffices to obtain \eqref{2.7} and then apply continuity arguments. Noticing  that our estimates  \eqref{ban}  is stronger than
\eqref{2.7},   we only need to prove \eqref{ban} which will be presented later. Hence we omit the details here.
\end{proof}

Similarly, for the problem \eqref{bao}-\eqref{bar}, we have
\begin{Proposition}\label{le2.8}
Let $\kappa(\theta)$ satisfying \eqref{bam} and $\epsilon\in(0,1)$. Let $u_0,\theta_0\in H^2$ and $\dv u_0=0$ in $\mathbb{R}^2$.
Then, for fixed $\epsilon>0$,  the problem \eqref{bao}-\eqref{bar} has a unique global solution $(u,\theta)$ satisfying
\begin{equation}
\|u\|_{L^\infty(0,T;H^2(\mathbb{R}^2) )}+\|u\|_{L^2(0,T;H^3(\mathbb{R}^2) )}+\|\theta\|_{L^\infty(0,T;H^2(\mathbb{R}^2) )}\leq C(\epsilon)\label{2.8}
\end{equation}
for some positive constant $C(\epsilon)$.
\end{Proposition}

In the subsequent sections, we use $C$ (independent of $\epsilon$ in Sections 5 and 6)
to denote the positive constant which may change from line to line.
 We also omit the spatial  domain $\mathbb{R}^3$  or $\mathbb{R}^2$  in the integrals below
for simplicity.

\medskip
\section{Proof of Theorem \ref{Ta}}

This section is devoted to the proof of Theorem \ref{Ta}, by Proposition \ref{le2.5}, we only need to prove \eqref{bah}.

First, it follows from  maximum principle and (\ref{bac}) and
(\ref{bad}) that
\begin{equation}
\|\theta\|_{L^\infty(0,T;L^\infty(\mathbb{R}^3))}\leq\|\theta_0\|_{L^\infty(\mathbb{R}^3)}\leq
C.\label{bca}
\end{equation}

 Multiplying  (\ref{bac}) by $\theta$, integrating the result over $\mathbb{R}^3$,  and using (\ref{bab}), we see that
$$\frac{1}{2}\frac{d}{dt}\int\theta^2dx+\int|\nabla\theta|^2dx=0,$$
which gives
\begin{equation}
\|\theta\|_{L^\infty(0,T;L^2(\mathbb{R}^3))}+\|\theta\|_{L^2(0,T;H^1(\mathbb{R}^3))}\leq
C.\label{bcb}
\end{equation}

Multiplying (\ref{baa}) by $u$, integrating the result over $\mathbb{R}^3$,  and using (\ref{bab}), (\ref{bca}),
(\ref{bcb}) and (\ref{bae}), we find that
$$\frac{1}{2}\frac{d}{dt}\int u^2dx+\frac{1}{C}\int|\nabla u|^2dx\leq
\int\theta e_3\cdot u dx\leq\|u\|_{L^2}\|\theta\|_{L^2}\leq
C\|u\|_{L^2},$$ which implies
\begin{equation}
\|u\|_{L^\infty(0,T;L^2(\mathbb{R}^3))}+\|u\|_{L^2(0,T;H^1(\mathbb{R}^3))}\leq C.\label{bcc}
\end{equation}

\medskip
\emph{Case I:  Assume that (\ref{baf}) holds.}

Applying the operator $\partial_i$ to (\ref{bac}), multiplying the result
by $(\partial_i\theta)^3$, integrating   over $\mathbb{R}^3$, summing over $i$, and using (\ref{bab}), the
H\"{o}lder and Gagliardo-Nirenberg inequalities, we have
\begin{align*}
&\frac{1}{4}\frac{d}{dt}\int|\nabla\theta|^4dx+\frac{3}{4}
\int|\nabla|\nabla\theta|^2|^2dx\nonumber\\
\leq &
C\int|u||\nabla\theta|^2|\nabla|\nabla\theta|^2|dx\\
 \leq&C\|u\|_{L^q}\||\nabla\theta|^2\|_{L^\frac{2q}{q-2}}
\|\nabla|\nabla\theta|^2\|_{L^2}\\
 \leq&C\|u\|_{L^q}\||\nabla\theta|^2\|_{L^2}^{1-\alpha_1}
\|\nabla|\nabla\theta|^2\|_{L^2}^{1+\alpha_1}\\
 \leq&\frac{1}{2}\|\nabla|\nabla\theta|^2\|_{L^2}^2+C\|u\|_{L^q}^p
\||\nabla\theta|^2\|_{L^2}^2\ \ \left(p=\frac{2}{1-\alpha_1}\right).
\end{align*}
Hence  it holds
\begin{equation}
\|\nabla\theta\|_{L^\infty(0,T;L^4(\mathbb{R}^3))}\leq C.\label{bcd}
\end{equation}

Multiplying (\ref{baa}) by $-\Delta u$,  integrating the result over $\mathbb{R}^3$, and using (\ref{bab}),
(\ref{bcd}), the H\"{o}lder and Gagliardo-Nirenberg inequalities, we
derive that
\begin{align*}
&\frac{1}{2}\frac{d}{dt}\int|\nabla u|^2dx+\frac{1}{C}\int|\Delta
u|^2dx\nonumber\\
\leq &\int(u\cdot\nabla)u\cdot\Delta u
dx+C\int|\nabla\theta||\nabla u||\Delta u|dx\\
 \leq&\|u\|_{L^q}\|\nabla u\|_{L^\frac{2q}{q-2}}\|\Delta
u\|_{L^2}+C\|\nabla\theta\|_{L^4}\|\nabla u\|_{L^4}\|\Delta
u\|_{L^2}\\
 \leq&C\|u\|_{L^q}\|\nabla u\|_{L^2}^{1-\alpha_1}\|\Delta
u\|_{L^2}^{1+\alpha_1}+C\|\nabla u\|_{L^4}\|\Delta u\|_{L^2}\\
 \leq&\frac{1}{2C}\|\Delta u\|_{L^2}^2+C\big\{\|u\|_{L^q}^p\|\nabla
u\|_{L^2}^2+ \|\nabla u\|_{L^2}^2\big\},
\end{align*}
which gives
\begin{equation}
\|u\|_{L^\infty(0,T;H^1(\mathbb{R}^3))}+\|u\|_{L^2(0,T;H^2(\mathbb{R}^3))}\leq C.\label{bce}
\end{equation}

Applying the operator $\Delta$ to (\ref{bac}), multiplying the result by $\Delta\theta$, integrating over $\mathbb{R}^3$, and
using (\ref{bab}), (\ref{bce}), the H\"{o}lder and
Gagliardo-Nirenberg inequalities, we obtain that
\begin{align*}
&\frac{1}{2}\frac{d}{dt}\int|\Delta\theta|^2dx+\int|\nabla\Delta\theta|^2dx
\nonumber\\
=& \sum\limits_{i=1}^3\int\partial_i(u\cdot\nabla\theta)\cdot\partial_i\Delta\theta
dx\\
 \leq&(\|u\|_{L^6}\|\Delta\theta\|_{L^3}+\|\nabla
u\|_{L^2}\|\nabla\theta\|_{L^\infty})\|\nabla\Delta\theta\|_{L^2}\\
 \leq&C(\|\Delta\theta\|_{L^3}+\|\nabla\theta\|_{L^\infty})\|\nabla\Delta\theta\|_{L^2}\\
 \leq&C\|\Delta\theta\|_{L^2}^{1/2}\|\nabla\Delta\theta\|_{L^2}^{3/2}
\leq\frac{1}{2}\|\nabla\Delta\theta\|_{L^2}^2+C\|\Delta\theta\|_{L^2}^2,
\end{align*}
which leads to
\begin{equation}
\|\theta\|_{L^\infty(0,T;H^2(\mathbb{R}^3))}+\|\theta\|_{L^2(0,T;H^3(\mathbb{R}^3))}\leq
C.\label{bcf}
\end{equation}

Applying the operator $\Delta$ to (\ref{baa}), multiplying the result  by $\Delta u$, integrating over   $\mathbb{R}^3$, and using
(\ref{bab}), (\ref{bae}), (\ref{bca}), (\ref{biln}), and (\ref{bce}), we
arrive at
\begin{align*}
&\frac{1}{2}\frac{d}{dt}\int|\Delta
u|^2dx+\frac{1}{C}\int|\nabla\Delta u|^2dx\\
 \leq&\sum\limits_{i=1}^3\int\partial_i(u\cdot\nabla
u)\cdot\partial_i\Delta u dx+C\int(|\Delta\mu(\theta)||\nabla
u|+|\nabla\mu(\theta)||\nabla^2u|)|\nabla\Delta u|dx\\
 \leq&(\|u\|_{L^6}\|\Delta u\|_{L^3}+\|\nabla u\|_{L^2}\|\nabla
u\|_{L^\infty})\|\nabla\Delta u\|_{L^2}\\
&+C(\|\Delta\theta\|_{L^3}\|\nabla
u\|_{L^6}+\|\nabla\theta\|_{L^\infty}\|\Delta
u\|_{L^2})\|\nabla\Delta u\|_{L^2}\\
 \leq&C(\|\Delta u\|_{L^3}+\|\nabla u\|_{L^\infty})\|\nabla\Delta
u\|_{L^2}\\
&+C(\|\Delta\theta\|_{L^3}+\|\nabla\theta\|_{L^\infty})\|\Delta
u\|_{L^2}\|\nabla\Delta u\|_{L^2}\\
 \leq&\frac{1}{2}\|\nabla\Delta u\|_{L^2}^2+C\|\Delta
u\|_{L^2}^2+C(\|\Delta\theta\|_{L^3}^2+\|\nabla\theta\|_{L^\infty}^2)\|\Delta
u\|_{L^2}^2,
\end{align*}
which implies that
\begin{equation}
\|u\|_{L^\infty(0,T;H^2(\mathbb{R}^3))}+\|u\|_{L^2(0,T;H^3(\mathbb{R}^3))}\leq C.\label{bcg}
\end{equation}
Thus  \eqref{bah} holds.

\medskip
\emph{Case II: Assume that (\ref{bag}) holds.}

Applying the operator $\partial_i$ to (\ref{bac}), multiplying the result by
$(\partial_i\theta)^3$, integrating over $\mathbb{R}^3$, summing over $i$, and using (\ref{bab}), the
H\"{o}lder and Gagliardo-Nirenberg inequalities, we obtain that
\begin{align*}
&\frac{1}{4}\frac{d}{dt}\int|\nabla\theta|^4dx
+\frac{3}{4}\int|\nabla|\nabla\theta|^2|^2dx\\
 \leq&C\int|\nabla u||\nabla\theta|^2\cdot|\nabla\theta|^2dx\\
 \leq&C\|\nabla u\|_{L^q}\||\nabla\theta|^2\|_{L^\frac{2q}{q-1}}^2\\
 \leq&C\|\nabla u\|_{L^q}\||\nabla\theta|^2\|_{L^2}^{2(1-\alpha_2)}
\|\nabla|\nabla\theta|^2\|_{L^2}^{2\alpha_2}\\
 \leq&\frac{1}{2}\|\nabla|\nabla\theta|^2\|_{L^2}^2+C\|\nabla
u\|_{L^q}^p\||\nabla\theta|^2\|_{L^2}^2\ \
\left(p=\frac{1}{1-\alpha_2}\right),
\end{align*}
which proves (\ref{bcd}).

Multiplying (\ref{baa}) by $-\Delta u$, integrating the result over $\mathbb{R}^3$, and using (\ref{bab}),
(\ref{bcd}), the H\"{o}lder and Gagliardo-Nirenberg inequalities, we
get
\begin{align*}
&\frac{1}{2}\frac{d}{dt}\int|\nabla u|^2dx+\frac{1}{C}\int|\Delta
u|^2dx\\
 \leq&\sum\limits_i\int u_i\partial_iu\partial_j^2u
dx+C\int|\nabla\theta||\nabla u||\Delta u|dx\\
 \leq&-\sum\limits_{i=1}^3\int\partial_ju_i\partial_iu\partial_ju
dx+C\|\nabla\theta\|_{L^4}\|\nabla u\|_{L^4}\|\Delta u\|_{L^2}\\
 \leq&C\|\nabla u\|_{L^q}\|\nabla u\|_{L^\frac{2q}{q-1}}^2+C\|\nabla
u\|_{L^4}\|\Delta u\|_{L^2}\\
 \leq&C\|\nabla u\|_{L^q}\|\nabla u\|_{L^2}^{2(1-\alpha_2)}\|\Delta
u\|_{L^2}^{2\alpha_2}+C\|\nabla u\|_{L^4}\|\Delta u\|_{L^2}\\
 \leq&\frac{1}{2C}\|\Delta u\|_{L^2}^2+C\big\{\|\nabla u\|_{L^q}^p\|\nabla
u\|_{L^2}^2+\|\nabla u\|_{L^2}^2\big\},
\end{align*}
which implies (\ref{bce}).

Noticing that the  calculations  for (\ref{bcf}) and (\ref{bcg}) still hold in this case, we hence
  complete  the proof of Theorem \ref{Ta}.

\medskip
\section{Proof of Theorem \ref{Tb}}

This section is devoted to the proof of Theorem \ref{Tb}, by Proposition \ref{le2.6}, we only need to prove \eqref{bah}.

First, by  maximum principle, it is easy to prove that
\begin{equation}
\|\theta\|_{L^\infty(0,T;L^\infty(\mathbb{R}^3))}\leq C.\label{bda}
\end{equation}
Noticing the  condition \eqref{bae}, we still have (\ref{bcb}) and (\ref{bcc}).

Similar to the Proposition 4.1 in \cite{WZ}, we can prove

\begin{Proposition}\label{Proa}
Let $u$ satisfy (\ref{baf}) and $\dv u=0$ in
$\mathbb{R}^3\times(0,\infty)$. Assume that $\theta\in
L^\infty(0,T;L^2)\cap L^2(0,T;H^1)$ is a weak solution to
(\ref{bac}) and (\ref{bad}). Then there exists a $\alpha\in (0,1)$ such that
$\theta\in C^\alpha([0,T]\times\mathbb{R}^3)$.
\end{Proposition}
\begin{proof}
 Most of the calculations are as same as that in
\cite{WZ}, the only difference is the calculations of the following
term
\begin{align*}
&\left|\int_{t_1}^{t_2}\int\theta_k^2 u\cdot\nabla\eta\ \eta dx
dt\right|\\
 \leq&\int_{t_1}^{t_2}\|u\|_{L^q}\|\eta\theta_k\|_{L^\frac{2q}{q-2}}
\|\theta_k\nabla\eta\|_{L^2}dt\\
 \leq&\int_{t_1}^{t_2}C\|u\|_{L^q}\|\eta\theta_k\|_{L^2}^{1-\alpha_1}
\|\nabla(\eta\theta_k)\|_{L^2}^{\alpha_1}\|\theta_k\nabla\eta\|_{L^2}dt\
\ \left(p=\frac{2}{1-\alpha_1}\right)\\
 \leq&\epsilon\int_{t_1}^{t_2}\|u\|_{L^q}^2\|\eta\theta_k\|_{L^2}^{2(1-\alpha_1)}
\|\nabla(\eta\theta_k)\|_{L^2}^{2\alpha_1}dt+C\|\theta_k\nabla\eta\|_{L^2(t_1,t_2;L^2(\mathbb{R}^3))}^2\\
 \leq&\epsilon\int_{t_1}^{t_2}\|u\|_{L^q}^p
dt\|\eta\theta_k\|_{L^\infty(t_1,t_2;L^2(\mathbb{R}^3))}^2\nonumber\\
&+C\epsilon\|\nabla(\eta\theta_k)\|_{L^2(t_1,t_2;L^2(\mathbb{R}^3))}^2
+C\|\theta_k\nabla\eta\|_{L^2(t_1,t_2;L^2(\mathbb{R}^3))}^2
\end{align*}
for any $0<\epsilon<1$. This completes the proof.
\end{proof}

Now, using an estimate of the gradient of  solution to the following
parabolic equation $$\partial_t\theta-\dv(\kappa(\theta)\nabla\theta)=-\dv(u\theta),$$ we have (see \cite{Am})
\begin{align}
\|\nabla\theta\|_{L^p(0,T;L^q(\mathbb{R}^3))} \leq &
C\|u\theta\|_{L^p(0,T;L^q(\mathbb{R}^3))}+C \nonumber\\
\leq&
C\|\theta\|_{L^\infty(0,T;L^\infty)}\|u\|_{L^p(0,T;L^q(\mathbb{R}^3))}+C\nonumber\\
\leq &
C\|u\|_{L^p(0,T;L^q(\mathbb{R}^3))}+C.\label{bdb}
\end{align}

Multiplying (\ref{baa}) by $-\Delta u$,  integrating the result over $\mathbb{R}^3$, and using (\ref{bab}),
(\ref{bdb}), the H\"{o}lder and Gagliardo-Nirenberg inequalities, we
have
\begin{align*}
&\frac{1}{2}\frac{d}{dt}\int|\nabla u|^2dx+\frac{1}{C}\int|\Delta
u|^2dx\\
 \leq&\int(u\cdot\nabla)u\cdot\Delta u dx+C\int|\nabla\theta||\nabla
u||\Delta u|dx\\
 \leq&\|u\|_{L^q}\|\nabla u\|_{L^\frac{2q}{q-2}}\|\Delta
u\|_{L^2}+C\|\nabla\theta\|_{L^q}\|\nabla
u\|_{L^\frac{2q}{q-2}}\|\Delta u\|_{L^2}\\
 \leq&C(\|u\|_{L^q}+\|\nabla\theta\|_{L^q})\|\nabla
u\|_{L^2}^{1-\alpha_1}\|\Delta u\|_{L^2}^{1+\alpha_1}\\
 \leq&\frac{1}{2C}\|\Delta
u\|_{L^2}^2+C(\|u\|_{L^q}^p+\|\nabla\theta\|_{L^q}^p)\|\nabla
u\|_{L^2}^2,
\end{align*}
which gives (\ref{bce}).

Applying the operator $\Delta$ to (\ref{bac}), multiplying the result by $\Delta\theta$,   integrating  over $\mathbb{R}^3$, and
using (\ref{bab}), (\ref{bce}), (\ref{biln}), (\ref{bdb}), the H\"{o}lder
and Gagliardo-Nirenberg inequalities, we have
\begin{align*}
&\frac{1}{2}\frac{d}{dt}\int|\Delta\theta|^2dx+\frac{2}{C}\int|\nabla\Delta\theta|^2dx\\
 \leq&\sum\limits_{i=1}^3\int\partial_i(u\cdot\nabla\theta)\cdot\partial_i\Delta\theta
dx+C\|\nabla\theta\|_{L^q}\|\Delta\theta\|_{L^\frac{2q}{q-2}}\|\nabla\Delta\theta\|_{L^2}\\
 \leq&(\|u\|_{L^6}\|\Delta\theta\|_{L^3}+\|\nabla
u\|_{L^2}\|\nabla\theta\|_{L^\infty})\|\nabla\Delta\theta\|+C\|\nabla\theta\|_{L^q}
\|\Delta\theta\|_{L^2}^{1-\alpha_1}\|\nabla\Delta\theta\|_{L^2}^{1+\alpha_1}\\
 \leq&C(\|\Delta\theta\|_{L^3}+\|\nabla\theta\|_{L^\infty})\|\nabla\Delta\theta\|_{L^2}
+C\|\nabla\theta\|_{L^q}^p\|\Delta\theta\|_{L^2}^2+\frac{1}{4C}\|\nabla\Delta\theta\|_{L^2}^2\\
 \leq&\frac{1}{2C}\|\nabla\Delta\theta\|_{L^2}^2+C\big\{\|\Delta\theta\|_{L^2}^2
+ \|\nabla\theta\|_{L^q}^p\|\Delta\theta\|_{L^2}^2\big\},
\end{align*}
which implies (\ref{bcf}).

Noticing the  calculations  for   (\ref{bcg}) still hold here, we hence
  complete  the proof of Theorem \ref{Tb}.

\medskip
\section{Proof of Theorem \ref{Tc}}

To complete our proof of Theorem \ref{Tc}, by Proposition \ref{le2.7}, we only need to prove a priori estimates (\ref{ban}).
By  maximum principle, it follows from (\ref{bak}) and
(\ref{bal}) that
\begin{equation}
\|\theta\|_{L^\infty(0,T;L^\infty(\mathbb{R}^2))}\leq\|\theta_0\|_{L^\infty(\mathbb{R}^2)}\leq
C.\label{bea}
\end{equation}

Multiplying (\ref{bak}) by $\theta$, integrating the result over $\mathbb{R}^2$, and using (\ref{baj}), (\ref{bam}),
and (\ref{bea}), we see that
$$\frac{1}{2}\frac{d}{dt}\int\theta^2dx+\int
\kappa(\theta)|\nabla\theta|^2dx=0,$$ which gives
\begin{equation}
\|\theta\|_{L^\infty(0,T;L^2(\mathbb{R}^2))}+\|\theta\|_{L^2(0,T;H^1(\mathbb{R}^2))}\leq
C.\label{beb}
\end{equation}

Multiplying (\ref{bai}) by $u$, integrating the result over $\mathbb{R}^2$, and using (\ref{baj}) and (\ref{beb}), we
find that
 $$
\frac{1}{2}\frac{d}{dt}\int u^2dx+\epsilon\int|\nabla
u|^2dx=\int\theta e_2 u dx\leq\|\theta\|_{L^2}\|u\|_{L^2}\leq
C\|u\|_{L^2},$$
which yields
\begin{equation}
\|u\|_{L^\infty(0,T;L^2(\mathbb{R}^2))}+\sqrt\epsilon\|u\|_{L^2(0,T;H^1(\mathbb{R}^2))}\leq
C.\label{bec}
\end{equation}

Multiplying (\ref{bai}) by $-\Delta u$, integrating the result over $\mathbb{R}^2$, and using (\ref{baj}), (\ref{beb})
and the fact that $$\int(u\cdot\nabla)u\cdot\Delta u dx=0, $$
 we get
$$\frac{1}{2}\frac{d}{dt}\int|\nabla u|^2dx+\epsilon\int|\Delta u|^2dx
\leq\int|\nabla\theta||\nabla u|dx\leq\|\nabla\theta\|_{L^2}\|\nabla
u\|_{L^2},$$ which leads to
\begin{equation}
\|u\|_{L^\infty(0,T;H^1(\mathbb{R}^2))}+\sqrt\epsilon\|u\|_{L^2(0,T;H^2(\mathbb{R}^2))}\leq
C.\label{bed}
\end{equation}

Then by taking the same calculations as those in \cite{WZ}, we obtain that
\begin{equation}
\|\theta\|_{C^\alpha(\mathbb{R}^3\times[0,T])}\leq C\label{bee}
\end{equation}
for some $\alpha\in(0,1)$ independent of $\epsilon>0$.

Now, using an estimate of the gradient of  solution of the
parabolic equation $$\partial_t\theta-\dv(\kappa(\theta)\nabla\theta)=-\dv(u\theta),$$ we obtain (see \cite{Am})
\begin{align}
\|\nabla\theta\|_{L^4(0,T;L^4(\mathbb{R}^2))} \leq&\|u\theta\|_{L^4(0,T;L^4(\mathbb{R}^2))}+C\nonumber\\
 \leq&\|u\|_{L^4(0,T;L^4(\mathbb{R}^2))}\|\theta\|_{L^\infty(0,T;L^\infty(\mathbb{R}^2))}+C
 \leq C.\label{bef}
\end{align}

Applying the operator $\Delta$ to (\ref{bai}), multiplying the result  by $\Delta\theta$,  integrating   over $\mathbb{R}^2$,  and
using (\ref{baj}), (\ref{bed}), the H\"{o}lder and Gagliardo-Nirenberg
inequalities, we derive that
\begin{align*}
&\frac{1}{2}\frac{d}{dt}\int|\Delta\theta|^2dx+\frac{1}{C}\int|\nabla
\Delta\theta|^2dx\\
 \leq&\sum\limits_{i=1}^2\int\partial_i(u\cdot\nabla\theta)\partial_i\Delta\theta
dx+C\|\nabla\theta\|_{L^4}\|\Delta\theta\|_{L^4}\|\nabla\Delta\theta\|_{L^2}\\
 \leq&C(\|u\|_{L^4}\|\Delta\theta\|_{L^4}+\|\nabla
u\|_{L^2}\|\nabla\theta\|_{L^\infty})\|\nabla\Delta\theta\|_{L^2}
+C\|\nabla\theta\|_{L^4}\|\Delta\theta\|_{L^4}\|\nabla\Delta\theta\|_{L^2}\\
 \leq&C(\|\Delta\theta\|_{L^4}+\|\nabla\theta\|_{L^\infty})\|\nabla
\Delta\theta\|_{L^2}+C\|\nabla\theta\|_{L^4}
\|\Delta\theta\|_{L^4}\|\nabla\Delta\theta\|_{L^2}\\
 \leq&C(\|\Delta\theta\|_{L^2}^{1/2}\|\nabla\Delta\theta\|_{L^2}^{1/2}
+\|\nabla\theta\|_{L^4}^{2/3}\|\nabla\Delta\theta\|_{L^2}^{1/3})
\|\nabla\Delta\theta\|_{L^2}\\
&+C\|\nabla\theta\|_{L^4}\|\Delta\theta\|_{L^2}^{1/2}
\|\nabla\Delta\theta\|_{L^2}^{3/2}\\
 \leq&\frac{1}{2C}\|\nabla\Delta\theta\|_{L^2}^2+C\big\{\|\Delta\theta\|_{L^2}^2
+\|\nabla\theta\|_{L^4}^2+\|\nabla\theta\|_{L^4}^4\|\Delta\theta\|_{L^2}^2\},
\end{align*}
which implies that
\begin{equation}
\|\theta\|_{L^\infty(0,T;H^2(\mathbb{R}^2))}+\|\theta\|_{L^2(0,T;H^3(\mathbb{R}^2))}\leq
C.\label{beg}
\end{equation}
Here we have used the estimate (see    \eqref{biln}):
\begin{equation}
\|\Delta \kappa(\theta)\|_{L^p}\leq C\|\Delta\theta\|_{L^p}\ \ \text{with}\ \
1<p<\infty.\label{beh}
\end{equation}

Denote $\omega=\cl u=\partial_1u_2-\partial_2u_1$ for
$u=(u_1,u_2)^{\mathrm{T}}$.
  Taking $\cl$ to (\ref{bai}) and using
(\ref{baj}), we infer that
\begin{equation}
\partial_t\omega+u\cdot\nabla\omega-\epsilon\Delta\omega=\partial_1\theta.\label{bei}
\end{equation}
Multiplying (\ref{bei}) by $|\omega|^{q-2}\omega$,   integrating the result over $\mathbb{R}^2$, and using (\ref{baj}) and
(\ref{beg}), we have $$\frac{1}{q}\frac{d}{dt}\int|\omega|^q
dx\leq\int\partial_1\theta|\omega|^{q-2}\omega
dx\leq\|\partial_1\theta\|_{L^q}\|\omega\|_{L^q}^{q-1},$$
 which gives
$$\frac{d}{dt}\|\omega\|_{L^q}\leq C\|\nabla\theta\|_{L^q},$$ and
thus
\begin{equation}
\|\omega\|_{L^q}\leq\|\omega_0\|_{L^q}+C\int_0^T\|\nabla\theta\|_{L^q}dt.\label{bej}
\end{equation}
Taking $q\rightarrow+\infty$ in (\ref{bej}), we obtain that
\begin{equation}
\|\omega\|_{L^\infty(0,T;L^\infty(\mathbb{R}^2))}\leq C.\label{bek}
\end{equation}

Applying the operator $\Lambda^3$ to (\ref{bai}),  multiplying the result by $\Lambda^3u$, integrating   over $\mathbb{R}^2$,
and using (\ref{baj}), (\ref{beg}), (\ref{bek}), (\ref{kpa}), and \eqref{log}, we conclude
that
\begin{align*}
 \frac{1}{2}\frac{d}{dt}\int|\Lambda^3u|^2dx
 \leq&-\int(\Lambda^3(u\cdot\nabla u)-u\nabla\Lambda^3u)\Lambda^3u
dx+\int\Lambda^3(\theta e_2)\cdot\Lambda^3 u dx\\
 \leq&C\|\nabla
u\|_{L^\infty}\|\Lambda^3u\|_{L^2}^2+C\|\Lambda^3\theta\|_{L^2}
\|\Lambda^3u\|_{L^2}\\
 \leq&C(1+\log(e+\|\Lambda^3u\|_{L^2}))\|\Lambda^3u\|_{L^2}^2
+C\|\Lambda^3\theta\|_{L^2}\|\Lambda^3u\|_{L^2},
\end{align*}
which leads to
\begin{equation}
\|u\|_{L^\infty(0,T;H^3 (\mathbb{R}^2))}\leq C.\label{bel}
\end{equation}
This completes the proof of Theorem \ref{Tc}.

\medskip
\section{Proof of Theorem \ref{Td}}

To complete our proof of Theorem \ref{Td}, by Proposition \ref{le2.8}, we only need to prove a priori estimates (\ref{bas}).
First, the estimate  (\ref{bea}) still holds. Similarly to (\ref{beb}) and (\ref{bec}), we have
\begin{align}
&\|\theta\|_{L^\infty(0,T;L^2 (\mathbb{R}^2))}+\sqrt\epsilon\|\theta\|_{L^2(0,T;H^1 (\mathbb{R}^2))}\leq
C,\label{bfa}\\
&\|u\|_{L^\infty(0,T;L^2 (\mathbb{R}^2))}+\|u\|_{L^2(0,T;H^1 (\mathbb{R}^2))}\leq C.\label{bfb}
\end{align}

Multiplying (\ref{bao}) by $-\Delta u$,   integrating the result over $\mathbb{R}^2$, and using (\ref{bap}), (\ref{bfa})
and the fact that
$$\int(u\cdot\nabla)u\cdot\Delta u dx=0,$$
we find that
\begin{align*}
\frac{1}{2}\frac{d}{dt}\int|\nabla u|^2dx+\int|\Delta
u|^2dx = -\int\theta e_2\Delta u dx
 \leq \|\theta\|_{L^2}\|\Delta u\|_{L^2}\leq C\|\Delta u\|_{L^2},
\end{align*}
which implies that
\begin{equation}
\|u\|_{L^\infty(0,T;H^1 (\mathbb{R}^2))}+\|u\|_{L^2(0,T;H^2 (\mathbb{R}^2))}\leq C.\label{bfc}
\end{equation}

By the Sobolev embedding theorem, it is easy to check that
\begin{equation}
\partial_tu-\Delta u+\nabla\pi=f,  \quad f:=\theta e_2-u\cdot\nabla u\in
L^2(0,T;L^p (\mathbb{R}^2)),\label{bfd}
\end{equation}
for any $p\in(2,\infty)$. Therefore, it follows from (\ref{bfd}) and
the regularity theory of Stokes equation that
\begin{equation}
\|u\|_{L^2(0,T;W^{2,p} (\mathbb{R}^2))}\leq C\|f\|_{L^2(0,T;L^p (\mathbb{R}^2))}+C\leq C.\label{bfe}
\end{equation}
And thus
\begin{equation}
\|\nabla u\|_{L^2(0,T;L^\infty (\mathbb{R}^2))}\leq C\|u\|_{L^2(0,T;W^{2,p} (\mathbb{R}^2))}\leq
C.\label{bff}
\end{equation}

Let $$\widetilde \kappa(\theta):=\int_0^\theta \kappa(\xi)d\xi.$$
We deduce from   (\ref{baq}) that
\begin{equation}
\partial_t\widetilde \kappa(\theta)+u\cdot\nabla\widetilde
\kappa(\theta)-\epsilon \kappa(\theta)\Delta\widetilde \kappa(\theta)=0.\label{bfg}
\end{equation}
Multiplying (\ref{bfg}) by $-\Delta\widetilde \kappa(\theta)$, integrating the result over $\mathbb{R}^2$, and  using
(\ref{bap}) and (\ref{bff}), we derive that
\begin{align*}
& \frac{1}{2}\frac{d}{dt}\int|\nabla\widetilde \kappa(\theta)|^2dx
+\epsilon\int \kappa(\theta)|\Delta\widetilde \kappa(\theta)|^2dx\\=& \int u\nabla\widetilde
\kappa(\theta)\Delta\widetilde \kappa(\theta) dx=-\int\sum\limits_{i,j}\partial_ju_i\partial_i\widetilde
\kappa(\theta)\partial_j\widetilde \kappa(\theta) dx
 \leq \|\nabla u\|_{L^\infty}\|\nabla\widetilde \kappa(\theta)\|_{L^2}^2,
\end{align*}
which implies
\begin{equation}
\|\theta\|_{L^\infty(0,T;H^1 (\mathbb{R}^2))}+\sqrt\epsilon\|\Delta\widetilde
\kappa(\theta)\|_{L^2(0,T;L^2 (\mathbb{R}^2))}\leq C.\label{bfh}
\end{equation}

Noting that (see \eqref{biln})
\begin{align}
\sqrt\epsilon\|\Delta\theta\|_{L^2(0,T;L^2 (\mathbb{R}^2))}=&\sqrt\epsilon
\|\Delta(\widetilde \kappa^{-1}\circ\widetilde
\kappa(\theta))\|_{L^2(0,T;L^2 (\mathbb{R}^2))}\nonumber\\
 \leq&\sqrt\epsilon C\|\Delta\widetilde \kappa(\theta)\|_{L^2(0,T;L^2 (\mathbb{R}^2))}\leq
C.\label{bfi}
\end{align}
By the Gagliardo-Nirenberg inequality, it follows from (\ref{bfh})
and (\ref{bfi}) that
\begin{align}
 \epsilon\|\nabla\theta\|_{L^4(0,T;L^4 (\mathbb{R}^2))}^4
 =&\epsilon\int_0^T\|\nabla\theta\|_{L^4}^4dt\leq C\epsilon\int_0^T
\|\nabla\theta\|_{L^2}^2\|\Delta\theta\|_{L^2}^2dt\nonumber\\
 \leq&C\epsilon\int_0^T\|\Delta\theta\|_{L^2}^2dt\leq C.\label{bfj}
\end{align}

Applying the operator $\partial_i$ to (\ref{baq}) gives
\begin{equation}
\partial_t\partial_i\theta+u\nabla\partial_i\theta+\partial_iu
\cdot\nabla\theta-\epsilon\dv\{\kappa(\theta)\nabla\partial_i\theta
+k'(\theta)\partial_i\theta\nabla\theta\}=0.\label{bfk}
\end{equation}
Multiplying (\ref{bfk}) by $(\partial_i\theta)^3$,  integrating the result over $\mathbb{R}^2$, summing over $i$, and
using (\ref{bap}), (\ref{bff}) and (\ref{bfj}), we obtain
\begin{align*}
&\frac{1}{4}\frac{d}{dt}\int|\nabla\theta|^4dx+\frac{\epsilon}{C}
\int|\nabla|\nabla\theta|^2|^2dx\\
 \leq&C\epsilon\int|\nabla\theta|\cdot|\nabla\theta|^2|\nabla|\nabla\theta|^2|dx
+C\|\nabla u\|_{L^\infty}\int|\nabla\theta|^4dx\\
 \leq&C\epsilon\|\nabla\theta\|_{L^4}\|(\nabla\theta)^2\|_{L^4}\|\nabla|\nabla
\theta|^2\|_{L^2}+C\|\nabla u\|_{L^\infty}\int|\nabla\theta|^4dx\\
 \leq&C\epsilon\|\nabla\theta\|_{L^4}\||\nabla\theta|^2\|_{L^2}^{1/2}
\cdot\|\nabla|\nabla\theta|^2\|_{L^2}^{3/2}+C\|\nabla
u\|_{L^\infty}\int|\nabla\theta|^4dx\\
 \leq&\frac{\epsilon}{2C}\|\nabla|\nabla\theta|^2\|_{L^2}^2+C\epsilon
\|\nabla\theta\|_{L^4}^4\||\nabla\theta|^2\|_{L^2}^2+C\|\nabla
u\|_{L^\infty}\int|\nabla\theta|^4dx,
\end{align*}
which gives
\begin{equation}
\|\nabla\theta\|_{L^\infty(0,T;L^4(\mathbb{R}^2))}\leq C.\label{bfl}
\end{equation}

Applying the operator $\Delta$ to (\ref{baq}), multiplying the result by $\Delta\theta$,  integrating over $\mathbb{R}^2$, and using
\eqref{biln}, (\ref{bfl}), (\ref{beh}), (\ref{bfe}), and (\ref{bff}), we conclude
that
\begin{align*}
&\frac{1}{2}\frac{d}{dt}\int|\Delta\theta|^2dx+\frac{\epsilon}{C}
\int|\nabla\Delta\theta|^2dx\\
 \leq&-\int(\Delta(u\cdot\nabla\theta)-u\nabla\Delta\theta)\Delta\theta
dx+C\epsilon\|\nabla\theta\|_{L^4}\|\Delta
k\|_{L^4}\|\nabla\Delta\theta\|_{L^2}\\
 \leq&C(\|\Delta u\|_{L^4}\|\nabla\theta\|_{L^4}+\|\nabla
u\|_{L^\infty}\|\Delta\theta\|_{L^2})\|\Delta\theta\|_{L^2}
+C\epsilon\|\Delta\theta\|_{L^4}\|\nabla\Delta\theta\|_{L^2}\\
 \leq&C(\|\Delta u\|_{L^4}+\|\nabla u\|_{L^\infty}
\|\Delta\theta\|_{L^2})\|\Delta\theta\|_{L^2}+\frac{\epsilon}{2C}
\|\nabla\Delta\theta\|_{L^2}^2+C\epsilon\|\Delta\theta\|_{L^2}^2,
\end{align*}
which leads to
\begin{equation}
\|\theta\|_{L^\infty(0,T;H^2 (\mathbb{R}^2))}\leq C.\label{bfm}
\end{equation}

Finally, applying the operator $\Delta$ to (\ref{bao}), multiplying the result by $\Delta u$, integrating over $\mathbb{R}^2$, and
using (\ref{bap}), (\ref{bfm}), and (\ref{bff}), we arrive at
\begin{align*}
&\frac{1}{2}\frac{d}{dt}\int|\Delta u|^2dx+\int|\nabla\Delta
u|^2dx\\
 =&\int\Delta\theta e_2\cdot\Delta u dx-\int\Delta(u\cdot\nabla
u)\cdot\Delta u dx\\
 \leq&\|\Delta\theta\|_{L^2}\|\Delta u\|_{L^2}+C\|\nabla
u\|_{L^\infty}\|\Delta u\|_{L^2}^2\\
 \leq&C\|\Delta u\|_{L^2}+C\|\nabla u\|_{L^\infty}\|\Delta
u\|_{L^2}^2,
\end{align*}
which gives
$$
\|u\|_{L^\infty(0,T;H^2(\mathbb{R}^2))}+\|u\|_{L^2(0,T;H^3(\mathbb{R}^2))}\leq
C.
$$
This completes the proof of Theorem \ref{Td}.


\medskip

{\bf Acknowledgements:}
 The authors are very grateful to the anonymous  referee for his/her constructive
comments and helpful suggestions, which  considerably improved the earlier version of this paper.
 Fan was supported by NSFC (Grant No. 11171154). Li was supported by NSFC (Grant No. 11271184, 10971094),
 NCET-11-0227, PAPD, and the Fundamental Research Funds for the Central Universities.




\begin{thebibliography}{99}

\bibitem{A} H. Abidi,    Sur l'unicit\'e pour le syst\`eme de {B}oussinesq avec
              diffusion non lin\'eaire, J. Math. Pures Appl.   91 (2009),   80-99.

\bibitem{AH} H. Abidi, T. Hmidi, On the global well-posedness for Boussinesq system, J. Differential Equations 233 (2007)
199-220.

\bibitem{ACWa} D. Adhikari, C. Cao, J. Wu, The 2D Boussinesq equations with vertical viscosity and vertical diffusivity,
J. Differential Equations 249 (2010) 1078-1088.

\bibitem{ACWb} D. Adhikari, C. Cao, J. Wu, Global regularity results for the 2D Boussinesq equations with vertical dissipation,
J. Differential Equations 251 (2011) 1637-1655.


\bibitem{Am} H. Amann, Linear and quasilinear parabolic problems.
vol. I, volume 89 of Monographs in Mathematics. Birkh\"{a}user
Boston Inc., Boston, MA, 1995.

\bibitem{BS} L. Brandolese, M. E. Schonbek,
Large time decay and growth for solutions of a viscous Boussinesq system. Trans. Amer. Math. Soc. 364 (2012), no. 10, 5057-5090.

\bibitem{Ch} D. Chae, Global regularity for the 2D Boussinesq
equations with partial viscosity terms, Adv. Math. 203(2006)
497-513.

\bibitem{CN} D. Chae, H.-S. Nam,
 Local existence and blow-up criterion for the Boussinesq equations.
 Proc. Roy. Soc. Edinburgh Sect. A 127 (1997), 935-946.

\bibitem{CW} D. Chae, J. Wu, The 2D Boussinesq equations with logarithmically
supercritical velocities,  Adv. Math. 230 (2012) 1618-1645.


\bibitem{DPa}  R. Danchin, M. Paicu, Global well-posedness issues for the inviscid Boussinesq system with Yudovich¡¯s type data,
Comm. Math. Phys. 290 (2009) 1-14.

\bibitem{DPb}   R. Danchin, M. Paicu, Global existence results for the anisotropic Boussinesq system in dimension two, Math.
Models Methods Appl. Sci. 21 (2011) 421-457.


\bibitem{DGa}
J.I. Diaz, G. Galliano, On the Boussinesq system with nonlinear thermal diffusion,
Nonlinear Anal., 30 (6) (1997), 3255-3263.

\bibitem{DGb}
J.I. Diaz, G. Galiano, Existence and uniqueness of solutions of the Boussinesq system with nonlinear thermal diffusion,
Topol. Methods Nonlinear Anal., 11 (1) (1998), 59-82.




\bibitem{DR} P. G. Drazin, W. H. Reid, Hydrodynamic Stability, Cambridge University Press, Cambridge, 1981.

\bibitem{FNW} J. Fan,  G. Nakamura, H. Wang,   Blow-up criteria of smooth solutions to the 3D Boussinesq system
with zero viscosity in a bounded domain. Nonlinear Anal. 75 (2012), no. 7, 3436-3442.

 \bibitem{FZ} J. Fan,  T. Ozawa,  Regularity criteria for the 3D density-dependent Boussinesq equations.
 Nonlinearity 22 (2009),  553-568.

\bibitem{Ga} E. Gagliardo, Propriet\`{a} di alcune classi di funzioni in pi\`{u}  variabili.
Ricerche di. Matem. 7(1958), 102-137.

\bibitem{GP} M. Grasselli, D. Prazak, Long time behavior of a
diffuse interface model for binary fluid mixtures with shear
dependent viscosity, Interfaces and Free Boundaries 13(2011)
507-530.

\bibitem{HK}
T. Hmidi, S. Keraani, On the global well-posedness of the Boussinesq system with zero viscosity, Indiana Univ.
Math. J. 58 (2009) 1591-1618.



\bibitem{HKRa} T. Hmidi, S. Keraani, F. Rousset, Global well-posedness for Euler¨CBoussinesq system with critical dissipation,
Comm. Partial Differential Equations 36 (2011) 420-445.

\bibitem{HKRb} T. Hmidi, S. Keraani, F. Rousset, Global well-posedness for a Boussinesq-Navier-Stokes system with critical
dissipation, J. Differential Equations 249 (2010) 2147-2174.

\bibitem{HR}
T. Hmidi,F. Rousset,   Global well-posedness for the Euler-Boussinesq system with axisymmetric data.
J. Funct. Anal. 260 (2011), no. 3, 745-796.

\bibitem{HH} P. C. Hohenberg, B. I. Halperin, Theory of dynamical
critical phenomena, Rev. Modern Phys. 49(1977) 435-479.



\bibitem{HL} T. Hou, C. Li, Global well-posedness of the viscous Boussinesq equations, Discrete Contin. Dyn. Syst. 12 (2005)
1-12.

\bibitem{KP} T. Kato, G. Ponce, Commutator estimates and the Euler
and Navier-Stokes equations, Comm. Pure Appl. Math. 41(1988)
891-907.

\bibitem{KOT} H. Kozono, T. Ogawa, Y. Taniuchi, The critical Sobolev
inequalities in Besov spaces and regularity criterion to some
semi-linear evolution equations. Math. Z. 242(2002) 251-278.

\bibitem{LPZ}
M.-J. Lai, R. Pan, K. Zhao,   Initial boundary value problem for two-dimensional viscous Boussinesq equations.
 Arch. Ration. Mech. Anal. 199 (2011), no. 3, 739-760.

\bibitem{LLT}
A. Larios, E. Lunasin, E.S. Titi, Global well-posedness for the 2D Boussinesq system without heat diffusion and
with either anisotropic viscosity or inviscid Voigt-$\alpha$ regularization, 2010. arXiv:1010.5024v1 [math.AP].

\bibitem{LBb} S. A. Lorca, J. L. Boldrini,
 The initial value problem for a generalized Boussinesq model: regularity and global existence of strong solutions.
 Mat. Contemp. 11 (1996), 71-94.

\bibitem{LB} S. A. Lorca, J. L. Boldrini, The initial value problem
for a generalized Boussinesq model, Nonlinear Anal. 36(1999)
457-480.

\bibitem{Ma} A.J. Majda, Introduction to PDEs and Waves for the Atmosphere and Ocean,   Courant Lecture Notes in
Mathematics, vol. 9, AMS/CIMS, 2003.

\bibitem{MX} C. Miao, L. Xue, On the global well-posedness of a class of Boussinesq-Navier-Stokes systems, NoDEA Nonlinear
Differential Equations Appl. 18 (6) (2011) 707-735.

\bibitem{Mi} J. M. Milhaljan, A rigorous exposition of the Boussinesq approximations applicable to a thin layer of fluid,
Astron. J. (3) 136 (1962) 1126-1133.

 \bibitem{Ni} L. Nirenberg,
On elliptic partial differential equations, Ann. Scuola Norm. Sup.
Pisa (3) 13(1959) 115-162.

 \bibitem{Pe} J. Pedlosky, Geophysical Fluid Dyanmics, Springer-Verlag, New York, 1987.

\bibitem{Tr} H. Triebel, Theory of function spaces, Monogr. Math.,
Birkh\"{a}user Verlag, Basel, Boston, 1983.


\bibitem{WZ} C. Wang, Z. Zhang, Global well-posedness for the 2-D
Boussinesq system with the temperature-dependent viscosity and
thermal diffusivity, Adv. Math. 228 (2011) 43-62.



\end{thebibliography}
\end{document}